\documentclass[reqno]{amsart}

\usepackage{bm}
\usepackage{amsmath}
\usepackage{amsfonts}
\usepackage{amssymb}
\usepackage{amsthm}
\usepackage{graphicx}
\usepackage{color}
\usepackage{enumitem}
\usepackage{url}

\theoremstyle{plain}

\newtheorem{theorem}{Theorem}[section]
\newtheorem{corollary}[theorem]{Corollary}
\newtheorem{definition}[theorem]{Definition}

\newtheorem{lemma}[theorem]{Lemma}

\newtheorem{remark}[theorem]{Remark}

\numberwithin{equation}{section}
\newcommand{\re}{{\mathbb R}}

\newcommand{\tr}{\mathop{\mathrm{tr}}}

\newcommand{\diag}{\mathop{\mathrm{diag}}}
\newcommand{\Real}{\mathop{\mathrm{Re}}}

\begin{document}

\date{4 June 2012}

\title[Affine property and jump behavior]{Affine processes on positive semidefinite $d\times d$ matrices have jumps of finite variation in dimension $d>1$}
\author[E.Mayerhofer]{Eberhard Mayerhofer}
\address{Deutsche Bundesbank, Wilhelm-Epstein-Stra{\ss}e 14, 60431 Frankfurt am Main, Germany}
\email{ eberhard.mayerhofer@gmail.com}
\thanks{E.M. is Marie--Curie Fellow at Deutsche Bundesbank. This research has obtained funding from WWTF (Vienna Science and Technology Fund) and from the European Community (FP7- Initial Training Network under grant agreement number PITN-GA-2009-237984). The funding is gratefully acknowledged.}

\begin{abstract}
The theory of affine processes on the space of positive semidefinite $d\times d$ matrices has
been established in a joint work with Cuchiero, Filipovi\'c and
Teichmann (2011). We confirm the conjecture stated therein that in dimension $d>1$
this process class does not exhibit jumps of infinite total
variation. This constitutes a geometric phenomenon which is in contrast to the
situation on the positive real line (Kawazu and Watanabe, 1974).
As an application we prove that the exponentially affine property of
the Laplace transform carries over to the Fourier-Laplace transform if the
diffusion coefficient is zero or invertible.
\end{abstract}


\maketitle %\tableofcontents

\section{Introduction}
Affine processes are a special class of stochastically continuous
Markov processes with the following feature: Some suitable integral
transform (such as the characteristic function \cite{dfs}, Laplace
transform (\cite{KW, cfmt}), Fourier-Laplace transform, or even
moment generating function \cite{ADPTA}) of their transition
function is exponentially affine in the state variable. It has
become custom  to describe affine processes in terms of a
parametrization of their infinitesimal generator -- quite similarly
 to the L\'evy class \cite{Sato1999}, where the so-called L\'evy-Khintchine triplet
 $(a,c, m(d\xi))$\footnote{For simplicity of notation only the one dimensional case is recalled here.} relative to a truncation function
 $\chi(\xi)$ allows a parametric description of the generator
\[
\mathcal A f(x)= a f''(x)+b f'(x)+\int\limits_{\mathbb R
\setminus\{0\}} (f(x+\xi)-f(x)-f'(x)\chi(\xi))m(d\xi).
\]
The affine property translates into affine drift, diffusive and jump
behavior, and the coefficients of the involved affine functions
determine the so-called ``admissible parameter set''(\cite{dfs}).
For instance, for the state space $\mathbb R_+:=[0,\infty)$,  Kawazu
and Watanabe \cite{KW} show that the infinitesimal generator of a
conservative affine processes $X$ takes the form \footnote{ Note
that in the case $\alpha=\beta=0,\mu=0$, $X$ is a L\'evy
subordinator.}
\[
\mathcal A f(x)= \alpha x f''(x)+ (b+\beta x)
f'(x)+\int\limits_{\mathbb R_+\setminus\{0\}}
(f(x+\xi)-f(x)-f'(x)\chi(\xi))(m(d\xi)+x\mu(d\xi)), \] with
``parameters'' $(\alpha\geq 0,b\geq 0,\beta\in\mathbb R,
m(d\xi),\mu(d\xi))$, where the last two objects are sigma-finite
measures on $\mathbb R_+\setminus\{0\}$ such that
\[
\int\limits_{\mathbb R_+\setminus\{0\}}(\|\xi\|\wedge 1)
m(d\xi)<\infty,\quad \int_{\mathbb
R_+\setminus\{0\}}(\|\xi\|^2\wedge 1) \mu(d\xi)<\infty.
\]
However, a full parametric characterization depends crucially on the
geometry of the state space, and the probabilistic properties of
affine processes may vary accordingly. Motivated by multivariate
extensions in the affine term structure literature as well as in
stochastic volatility, Duffie, Filipovi\'c and Schachermayer
\cite{dfs} establish a unified theory on the so-called canonical
state spaces $\mathbb R_+^m\times\mathbb R^n$ (for further insights,
and certain simplifications, see \cite{keller, kst}). The recent
theory of Cuchiero, M., Filipovi\'c and Teichmann \cite{cfmt} for
affine processes on positive semidefinite matrices $S_d^+$ is a
response to suggestions in the finance literature concerning affine
multi-asset models based on matrix factors. Those, in turn, have
mostly used the class of Wishart processes as put forward in
\cite{bru}, or the OU-type processes driven by matrix variate L\'evy
subordinators \cite{barndorffstelzer}. For a review on financial
modeling issues with matrix factors, see the extensive introduction
of \cite{cfmt}, as well as the references given therein.

Aim of this paper is to show that affine processes on $S_d^+$,
$d\geq 2$, do not exhibit jumps of infinite total variation (Theorem
\ref{th: maintheorem}). This important result confirms a conjecture
formulated in \cite[Section 2.1.4]{cfmt}. In the conservative case,
it allows to simplify the semimartingale decomposition of
\cite[Theorem 2.6]{cfmt}; this is subject of Theorem
\ref{semimartingale}. A crucial application of Theorem \ref{th:
maintheorem} concerns the affine character of the Fourier-Laplace
transform of affine processes (Theorem \ref{th41}). In particular,
we show that in the presence of non-degenerate diffusion components,
affine processes are affine in the sense of Duffie, Filipovi\'c and
Schachermayer (Theorem \ref{th42} and Corollary \ref{cor42}). This means, that
their characteristic function is exponentially affine in the state-variable.
A detailed introduction to this topic with technical remarks is given
in Section \ref{sec: FLT}.

The main result of this paper, Theorem \ref{th41}, reveals a geometric phenomenon;
Indeed, in the much simpler case $d=1$, where the state space
simplifies to the positive real line $\mathbb R_+$, stochastic
processes with jumps of infinite total variation actually exist. For
instance, let $d\xi$ denote the Lebesgue measure on $[0,1]$ and
define a linear jump characteristic as
\[
\mu(d\xi):=\xi^{-2}1_{(0,1]}(\xi)d\xi
\]
Clearly $1=\int_0^1 \xi^2 \mu(d\xi)<\infty$, hence due to \cite{KW},
an affine pure jump process $X$ with infinitesimal generator
\[
\mathcal Af(x):=x\int\limits_{\mathbb R_+ \setminus\{0\}}
\left(f(x+\xi)-f(x)-f'(x)\xi\right) \mu(d\xi)
\]
exists. Let us start $X$ at $x>0$ and denote by $X_t$ its c\`adl\`ag
representation\footnote{Such exists due to the Feller property of
$X$, see \cite{dfs}.}. Then $X_t$ is a special semimartingale with
characteristics $(A=0, B=0, \nu(dt,d\xi))$, where the compensator of
$X_t$ equals
\[
\nu(dt, d\xi)=X_t \mu(d\xi)dt.
\]
The canonical decomposition of $X_t$ is given in terms of the
Poisson random measure associated with its jumps, $\mu^X(dt,d\xi)$:
\[
X_t=x+\xi\star (\mu^X-\nu)=\lim_{\varepsilon\downarrow
0}\int\limits_{\xi>\varepsilon}\int\limits_0^t \xi
(\mu^X(ds,d\xi)-\nu(ds,d\xi))
\]
and clearly $X_t>0$ a.s., because the jumps of $X$ are positive
throughout. Hence, almost surely it holds that\footnote{Of course in the
finite case, the two summands would differ in general.}:
\[
\sum_{s\leq t} \vert\Delta X_s\vert=\sum_{s\leq t} \Delta X_s=\int
\limits_{\mathbb R_+\setminus\{0\}} \xi\,\mu(d\xi)\times
\left(\int_0^t X_s ds\right)=\infty
\]
For $d\geq 2$ the complex geometry of the boundary $\partial S_d^+$
of $S_d^+$ -- it is not anymore the origin only, nor it is a smooth
manifold -- leads to non-trivial restrictions concerning the linear
jump behavior. One of these is \eqref{eq: mu} below, which
expresses that transversal to $\partial S_d^+$ only finite variation
jumps are allowed. In addition, there is a non-trivial tradeoff
between linear drift and linear jumps, see eq.~\eqref{eq:
inwarddrift}. One of the nice consequences of Theorem \ref{th:
maintheorem} is that these two conditions may be disentangled from
each other, into a simple condition that the drift must be inward
pointing at the boundary (eq.~\eqref{eq: betaij}) and the
compensator of the affine processes satisfies a stronger
integrability condition (see \eqref{eq: mu}). Furthermore, the
admissible parameter set is now formulated independent to truncation
functions, which is impossible in the setting of canonical state
spaces \cite{dfs}, and in particular for $d=1$.

It should perhaps be noted that the original characterization of
affine processes \cite[Theorem 2.4]{cfmt} and all consequences
thereof are stated in a way, which nest the one-dimensional one
(cf.~ \cite{KW} and \cite{dfs}) -- this is possible in view of the
implicit nature of condition \eqref{eq: inwarddrift}. As such the
preceding theory is perfectly valid in its original formulation; the
contribution of the present work, however, is a technical
simplification of the theory of affine processes on the cone of
positive semidefinite matrices $S_d^+$ of arbitrary dimension $d\geq
2$, and the additional theoretic results concerning the
Fourier-Laplace transform of this process class.

\section{Notation and definition of the affine property}
We try to keep notation and presentation of this paper as simple as
possible. As reference, both for applied and theoretic issues, see
the quite extensive work \cite{cfmt}; this also concerns technically
involved facts, which are here only recollected in prose.

$1_{A}$ equals the indicator
function corresponding to some set $A$. $S_d$ denotes the linear space of real $d\times d$
symmetric matrices, and $\langle x,y\rangle:=\tr(xy)$ is the standard scalar
product thereon, given by the trace of the matrix product. Accordingly,
$\|\,\|$ is the induced norm on $S_d$, and the pierced unit ball equals
\[
B^\circ_1:=\{z\in S_d^+\mid 0<\|z\|\leq 1\}.
\]
 The natural order introduced by the closed convex cone $S_d^+$ is denoted by
$\preceq$. The cone of positive definite matrices is denoted by $S_d^{++}$
(and clearly is the interior of $S_d^+$). $\nabla f(x)$ denotes the Frechet derivative of a
function $f$ at $x\in S_d$. This coordinate free notation allows much a much
shorter and more elegant presentation; for an account of the involved details and the
coordinate wise way to write what follows, the reader is referred to the nice
exhibition \cite{alfonsi} as well as \cite{cfmt}. Only in the proof
of the main theorem \ref{th: maintheorem} coordinates are used.

We consider a time-homogeneous Markov process $X$ with state space
$S_d^+$ and semigroup $(P_t)_{t \geq 0}$ acting on bounded Borel
measurable functions $f$
\begin{align*}
P_tf(x)=\int\limits_{S_d^+} f(\xi)p_t(x,d\xi), \quad x \in S_d^+.
\end{align*}
Here $p_t(x,d\xi)$ denotes the (possibly sub--)Markovian transition function of $X$.

\begin{definition}\label{def: affine process Sd+}
The Markov process $X$ is called \emph{affine} if
\begin{enumerate}
  \item it is stochastically continuous, that is, $\lim_{s\to t}
p_s(x,\cdot)=p_t(x,\cdot)$ weakly on $S_d^+$ for every $t$ and $x\in
S_d^+$, and
\item\label{def: affine process Sd+2} its Laplace transform has exponential-affine
dependence on the initial state:
\begin{align}\label{def: affine process}
P_te^{-\langle u, x\rangle}=\int\limits_{S_d^+}e^{-\langle u, \xi
\rangle}p_t(x,d\xi)=e^{-\phi(t,u)-\langle \psi(t,u),x\rangle},
\end{align}
for all $t$ and $u,x\in S_d^+$, for some functions $\phi: \re_+
\times S_d^+ \rightarrow \re_+$ and $\psi: \re_+ \times S_d^+
\rightarrow S_d^+$.
\end{enumerate}
\end{definition}
\section{Main Result and Proof}
The so-called {\it
admissible parameter set} is introduced in the following. Note that unlike \cite[Definition
2.3]{cfmt} truncation functions may be omitted, and the complicated
admissibility condition (\cite[(2.11)]{cfmt}, see also \eqref{eq:
inwarddrift} in the proof below) is dropped:

\begin{definition}\label{def: necessary admissibility}
Let $d\geq 2$. An \emph{admissible parameter set} $(\alpha, b,B, c, \gamma,
m(d\xi), \mu(d\xi))$ consists of
\begin{itemize}
  \item a linear diffusion coefficient $\alpha\in S_d^+$,
  \item a constant drift term $b\in S_d^+$ which satisfies
  \begin{equation}\label{eq: alpha}
    b \succeq (d-1)\alpha,
  \end{equation}
  \item a constant killing rate term $c \in \re_+$,
  \item a linear killing rate coefficient $\gamma \in S_d^+$,
\item a constant jump term: a Borel measure $m$ on $S_d^+\setminus\{0\}$ satisfying
\begin{equation}\label{eq: m}
   \int\limits_{S_d^+\setminus\{0\}}  \left(\|\xi\|\wedge 1\right)
    m(d\xi) <
\infty,
  \end{equation}
 \item
a linear jump coefficient $\mu$ which is an $S_d^+$-valued\footnote{We deviate here
a little from the corresponding definition in \cite{cfmt}, where $\mu$ is a finite measure on $S_d^+\setminus\{0\}$ (later divided by $\|\xi\|^2\wedge 1$). Here ''$S_d^+$--valued`` has to be understood as follows: For any Borel set $E$ in $S_d^+$ such that its $S_d^+$--topological closure $\bar E\subset S_d^+\setminus\{0\}$ we have $\mu(E)\in S_d^+$. This allows infinite activity jumps with state-dependent intensity: Indeed, there exist $\mu$ for which $\mu(S_d^+\setminus\{0\})=\infty$, which nevertheless satisfy eq.~\eqref{eq: mu}. The latter simply means that $(\|\xi\|\wedge 1) \mu(d\xi)$ is a finite measure.}, sigma
finite measure on $S_d^+\setminus\{0\}$ and \footnote{The
integral is of course matrix valued, and $<\infty$ expresses that
    it is finite.}
\begin{equation}\label{eq: mu}
   \int\limits_{S_d^+\setminus\{0\}}  \left(\|\xi\|\wedge 1\right)
    \mu(d\xi) <\infty.
  \end{equation}
\item and, finally, a linear drift $B$, which is a linear map from $S_d$ to $S_d$ and ``inward pointing'' at the boundary.
That is,
\begin{equation}\label{eq: betaij}
\langle B(x),u\rangle \geq 0\quad\text{for all $x,u\in S_d^+$ with $\langle x,
u\rangle=0$.}
\end{equation}
\end{itemize}
\end{definition}

The main statement of this paper follows:
\begin{theorem}\label{th: maintheorem}
Let $X$ be an affine processes on $S_d^+$ ($d\geq 2$). Then its
Markovian semigroup $(P_t)_t$ has an infinitesimal generator
$\mathcal A$  acting on the
space of rapidly decreasing functions\footnote{For further details, see \cite[Appendix B]{cfmt}.} supported on $S_d^+$
\begin{align*}
\mathcal Af(x)&=2\langle\nabla \alpha\nabla f(x),x\rangle+ \langle
\nabla f(x), b+B(x)\rangle- (c+\langle \gamma,x\rangle)f(x)\\&+
\int\limits_{S_d^+ \setminus\{0\}}
\left(f(x+\xi)-f(x)\right)(m(d\xi)+\langle \mu(d\xi), x\rangle).
\end{align*}
where $(\alpha,b,B,c,\gamma,m,\mu)$ are an admissible parameter set in the sense of Definition \ref{def: necessary admissibility}.
\end{theorem}
\begin{proof}
Let $\chi(\xi): S_d^+\rightarrow S_d^+$ be a truncation function,
that is $\chi(\xi)=\xi$ near zero, and $\chi$ is continuous, and
(what may be assumed without loss of generality) bounded by $1$. By
\cite[Theorem 2.4 and Definition 2.3]{cfmt} the semigroup $(P_t)_t$
has an infinitesimal generator $\mathcal A$ acting as
\newline
\begin{align*}
\mathcal Af(x)&=2\langle\nabla \alpha\nabla f(x),x\rangle+ \langle
\nabla f(x), b+\widetilde B(x)\rangle\\\quad&- (c+\langle
\gamma,x\rangle)f(x)+ \int\limits_{S_d^+ \setminus\{0\}}
\left(f(x+\xi)-f(x)\right) m(d\xi)\\\quad&+ \int\limits_{S_d^+
\setminus\{0\}} \left(f(x+\xi)-f(x)-\langle \chi(\xi),\nabla
f(x)\rangle\right)\langle \mu(d\xi), x\rangle.
\end{align*}
where $c\geq 0, \gamma\in S_d^+$ and the parameters $\alpha,b\in
S_d^+$ satisfy \eqref{eq: alpha} and $\mu(d\xi)$ is a sigma finite
$S_d^+$ valued measure on $S_d^+\setminus\{0\}$ which integrates
$\|\xi\|^2\wedge 1$. Furthermore, $m(d\xi)$ is a Borel measure on
$S_d^+\setminus\{0\}$ which satisfies \eqref{eq: m}. We also know
from \cite[Theorem 2.4 and Definition 2.3]{cfmt} that jumps of
infinite total variation may only occur parallel to the boundary. In
terms of admissibility conditions of the parameters, this is
expressed in eq.~\eqref{eq: m} as well as the following condition:
\begin{equation}\label{eq: mu*}
   \int\limits_{S_d^+\setminus\{0\}} \langle \xi, u\rangle \langle \mu(d\xi),x\rangle <
\infty,\quad\text{ $x,u\in S_d^+$ with $\langle x, u\rangle=0$.}
  \end{equation}
Furthermore, the drift  must be inward pointing at the boundary.
That is expressed in the terms of the positivity of the constant
drift $b\in S_d^+$ (as above), as well as the following constraint
on the linear part $\widetilde B$ (a linear map from $S_d$ to
$S_d$):
\begin{equation}\label{eq: inwarddrift}
\langle \widetilde B(x),u\rangle -
\int\limits_{S_d^+\setminus\{0\}}\left\langle\chi(\xi),u\right\rangle
\langle \mu(d\xi),x\rangle \geq 0,\quad\text{ $x,u\in S_d^+$ with
$\langle x, u\rangle=0$.}
\end{equation}
Note that $\langle x,u\rangle=0$ is equivalent to $xu=ux=0$, that is
$x,u\in\partial S_d^+$, see also \cite[Lemma 4.1 (i)--(iii)]{cfmt}.

Suppose for a moment that the validity of \eqref{eq: mu} had already been shown. Then
$\mu$ integrates $\|\chi(\xi)\|\leq \|\xi\|\wedge 1$ and therefore
a new drift may be introduced as
\begin{equation}\label{the new drift}
B(x):=\widetilde B(x)-\int\limits_{S_d^+\setminus\{0\}}\chi(\xi)
\langle \mu(d\xi),x\rangle \end{equation} which, in view of
\eqref{eq: inwarddrift} satisfies admissibility condition \eqref{eq:
betaij}. Hence the proof of the theorem were settled. So the essential point of the statement
is \eqref{eq: mu}. We use standard Euclidean coordinates on $S_d$
for the remainder of the proof; all indices range between $1$ and
$d$, if not otherwise stated. Let $\{c^{ij},i\leq j\}$ denote the
canonical basis of $S_d$, that is, the $kl$th coefficient of
$c^{ij}$ equals
\[
c^{ij}_{kl}=\delta_{ik}\delta_{jl}+\delta_{jk}\delta_{il}(1-\delta_{ij}),
\]
where $\delta_{ij}$ denotes the Kronecker delta. If $i=j$ then
$c^{ii}_{kl}=\delta_{ik}\delta_{il}$ (the diagonal matrix with a $1$
in the $i$th diagonal entry and zeros everywhere else). Otherwise
$c^{ij}$ is zero except in the $ij$th and $ji$th entry, where it is
equal $1$. We may evaluate $\mu$ coordinate wise as
$\mu(A)=(\mu_{ij}(A))_{ij}$, $A\in S_d$ such that $0\notin\bar A$,
the latter denoting the topological closure of the set $A$. Let
$c_i^*:=1-c^{ii}$, where $I$ is the unit matrix. Then clearly
$\langle c^{ii}, c_i^*\rangle=0$, and by eq.~\eqref{eq: mu*} it
holds that
\begin{equation}\label{eq: mu1}
   \int\limits_{B^\circ_1} \xi_{ii} \mu_{jj}(d\xi) <
\infty,\quad\text{ $1\leq i,j\leq d$, $i\neq j$}.
  \end{equation}
We show now that a similar integrability condition must also hold
for $1\leq i=j\leq d$. To circumvent integrability issues at the
origin, the measure $\mu$ is pierced as follows near $0$: For $\epsilon>0$ we introduce the new (and by construction
finite measures) $\mu^{\varepsilon}(d\xi):=\mu(d\xi)
1_{\varepsilon<\|\xi\|\leq 1}(d\xi)$. In particular,
\[
\mu_{ij}^\varepsilon(d\xi)=\mu_{ij} 1_{\varepsilon<\|\xi\|\leq 1
}(d\xi)
\]
are all signed finite measures $(1\leq i,j\leq d)$  such that for
all $\varepsilon>0$ we have
\begin{equation}\label{eq 1}
-\infty<
\int\limits_{S_d^+}\xi_{kl}\mu_{ij}^\varepsilon(d\xi)<\infty,\quad
\text{ $1\leq i,j,k,l\leq d$}.
\end{equation}
By \eqref{eq: mu1}, there exists a positive constant $M$ such that
for all $\varepsilon>0$
\begin{equation}\label{eq 2}
0\leq \int\limits_{S_d^+}\xi_{ii}\mu_{jj}^\varepsilon(d\xi)<M,\quad
i\neq j.
\end{equation}
We introduce now the following boundary points of $S_d^+$.
\[
e_\pm^{ij}:= c^{ii}\pm c^{ij}+c^{jj},\quad i\neq j
\]
By construction $\langle e_+^{ij},e_-^{ij}\rangle=0$. Setting
$x=e_+^{ij}$ and $u=e_-^{ij}$ and applying \eqref{eq: mu*}, we must
have
\begin{equation*}\label{eq: mu3a}
0\leq\int\limits_{B^\circ_1}
(\xi_{ii}-2\xi_{ij}+\xi_{jj})(\mu_{ii}(d\xi)+2\mu_{ij}(d\xi)+\mu_{jj}(d\xi))
< \infty,\quad\text{ $i\neq j$}.
\end{equation*}
Similarly, we obtain by using $x=e_-^{ij}$ and $u=e_+^{ij}$ that
\begin{equation*}\label{eq: mu3b}
0\leq \int\limits_{B^\circ_1}
(\xi_{ii}+2\xi_{ij}+\xi_{jj})(\mu_{ii}(d\xi)-2\mu_{ij}(d\xi)+\mu_{jj}(d\xi))
< \infty,\quad i\neq j.
\end{equation*}
Accordingly, there exists a constant positive $M_1$ we have, for all
$\varepsilon>0$,
\begin{equation}\label{eq: mu4a}
0\leq\int\limits_{B^\circ_1}
(\xi_{ii}-2\xi_{ij}+\xi_{jj})(\mu_{ii}^\varepsilon(d\xi)+2\mu_{ij}^\varepsilon(d\xi)+\mu_{jj}^\varepsilon(d\xi))
<M_1,\quad i\neq j.
\end{equation}
and
\begin{equation}\label{eq: mu4b}
0\leq\int\limits_{B^\circ_1}
(\xi_{ii}+2\xi_{ij}+\xi_{jj})(\mu_{ii}^\varepsilon(d\xi)-2\mu_{ij}^\varepsilon(d\xi)+\mu_{jj}^\varepsilon(d\xi))
< M_1,\quad i\neq j.
\end{equation}
Summing up \eqref{eq: mu4a} and \eqref{eq: mu4b} we therefore obtain
\begin{align*}
0\leq\int\limits_{B^\circ_1}&
\left(\xi_{ii}\mu_{ii}^\varepsilon(d\xi)-2\xi_{ij}\mu_{ij}^\varepsilon(d\xi)+\xi_{jj}\mu_{jj}^\varepsilon(d\xi)\right)\\\nonumber
&+\int\limits_{B^\circ_{\leq
1}}\left(\xi_{ii}\mu_{jj}^\varepsilon(d\xi)-2\xi_{ij}\mu_{ij}^\varepsilon(d\xi)+\xi_{jj}\mu_{ii}^\varepsilon(d\xi)\right)<2M_1,
\end{align*}
for all $i\neq j$. The two integrals are non-negative, because
$\mu^\varepsilon$ is an $S_d^+$ valued measure. We therefore
conclude that both of them are finite:
\begin{align}\label{eq: mu5a}
0\leq\int\limits_{B^\circ_1}&
\left(\xi_{ii}\mu_{ii}^\varepsilon(d\xi)-2\xi_{ij}\mu_{ij}^\varepsilon(d\xi)+\xi_{jj}\mu_{jj}^\varepsilon(d\xi)\right)<2M_1,\quad
i\neq j\\\label{eq: mu5b} 0\leq\int\limits_{B^\circ_{\leq
1}}&\left(\xi_{ii}\mu_{jj}^\varepsilon(d\xi)-2\xi_{ij}\mu_{ij}^\varepsilon(d\xi)+\xi_{jj}\mu_{ii}^\varepsilon(d\xi)\right)<2M_1,\quad
i\neq j
\end{align}
By subtracting \eqref{eq 2} from \eqref{eq: mu5b} twice, once for
$i,j$ and then $j,i$, we have
\[
-M_1<\int\limits_{B^\circ_1}\xi_{ij}\mu_{ij}^\varepsilon(d\xi)<M,\quad
i\neq j
\]
for all $\varepsilon>0$. Plugging this information back into
\eqref{eq: mu5a} and using the fact that $\xi_{ii}\geq 0$, and
$\mu^\varepsilon$ is positive semidefinite, we obtain
\[
0\leq\int\limits_{B^\circ_1}
\left(\xi_{ii}\mu_{ii}^{\varepsilon}(d\xi)+\xi_{jj}\mu_{jj}^\varepsilon(d\xi)\right)<2(M_1+M).
\]
The choices of $i$ was arbitrary. Taking into account \eqref{eq:
mu1} and the preceding uniform estimate in $\varepsilon$, we finally
conclude
\begin{equation}\label{eq: mufinal}
  0\leq  \int\limits_{B^\circ_1} \xi_{ii} \mu_{jj}(d\xi) <
\infty,\quad 1\leq i,j\leq d.
  \end{equation}
Define the positive measure $\tr(\mu)(d\xi)$ on Borel sets $A$ with
$0\notin \bar A$ by $\tr(\mu(A))$. Eq.~\eqref{eq: mufinal} implies
immediately
\begin{equation}\label{eq: mufinal1}
   \int\limits \limits_{B^\circ_1} \tr(\xi)\tr(\mu)(d\xi) <
\infty.
  \end{equation}
We finally show the admissibility condition \eqref{eq: mu}:  Let $\xi$ be a positive semidefinite matrix with diagonalization $\xi=UD U^\top$, where $U$ is orthogonal and $D=\diag(\lambda_1,\dots,\lambda_d)$.
By using this diagonalization and the cyclic property of the trace, we obtain
\begin{align}\label{eq: normestimate}
\|\xi\|^2&=\tr(UD^2U)=\tr(D^2)=\sum_{i=1}^d \lambda_i^2 \leq (\sum_{i=1}^d \lambda_i)^2\\\nonumber&=(\tr D)^2=(\tr UD U)^2=\tr(\xi)^2,
\end{align}
where $\leq $ follows from the non-negativity of the eigenvalues $\lambda_i$. Using this technical detail,
we infer from \eqref{eq: mufinal1} the following estimate:
\begin{equation}\label{finite ex}
\int_{S_d^+\setminus\{0\}}(\|\xi\|\wedge 1)\tr (\mu)(d\xi)<\infty.
\end{equation}
By Lemma \ref{norm versus trace lemma} below we may conclude the validity of condition \eqref{eq: mu}. Hence the definition of $B$ by eq.~\eqref{the new drift} is legitimate.
\end{proof}
The following technical statement has just been used and again will be used in the proof of Theorem \ref{th41}:
\begin{lemma}\label{norm versus trace lemma}
For any non-negative Borel--measurable function $g$ we have
\[
\|\int_{S_d^+\setminus\{0\}} g(\xi)\mu(d\xi)\|\leq \int_{S_d^+\setminus\{0\}} g(\xi)\tr (\mu)(d\xi)\leq \infty.
\]
\end{lemma}
\begin{proof}
Since $\mu$ is a positive semidefinite measure, we have
by eq.~\eqref{eq: normestimate} for any Borel set $B\subset S_d^+\setminus\{0\}$ the estimate $\|\mu(B)\|\leq \tr(\mu(B))$. Hence approximating the function $g$ by non-negative simple functions, the assertion follows. In particular, the integral \eqref{eq: mu} must be finite whenever \eqref{finite ex} is finite.
\end{proof}

\subsection{The Semimartingale Decomposition} Suppose $X$ is a conservative
\footnote{For sufficient and necessary condition of
conservativeness, see \cite[Remark 2.5]{cfmt} and \cite[Section
3]{may_smi_09}.} affine process on $S_d^+$, $d\geq 2$, with
admissible parameter set $(\alpha, b, B,0,0, m(d\xi), \mu(d\xi))$.
In view of the Feller property \cite[Theorem 2.4]{cfmt} of $X$, for
each initial state $x\in S_d^+$, there exists a modification of
$X_t:=(X^x)_{t\geq 0}$ on the canonical path space, which is a
c\`adl\`ag semimartingale. Since we know that the jumps of $X$ are
of total finite variation, we have as an immediate consequence of
\cite[Theorem 2.6]{cfmt},
\begin{theorem}\label{semimartingale}
Let $\Sigma$ be a $d\times d$ matrix such that
$\Sigma^\top\Sigma=\alpha$. Then there exists, possibly on an
enlargement of the probability space, a $d \times d$-matrix $W$ of
standard Brownian motions such that $X$ admits the following
representation
\[
X_t=x+ bt+\int\limits_0^t
B(X_{s})ds+\int\limits_0^t\left(\sqrt{X_{s}}dW_s\Sigma+\Sigma^{\top}dW_s\sqrt{X_{s}}\right)
+\int_0^t\int_{S_d^+}\xi\,\mu^X(ds,d\xi),
\]
where $\mu^X(d\xi,ds)$ is the random measure associated with the
jumps of $X$, having compensator
\[
\nu(dt,d\xi)=(m(d\xi)+\langle X_t,\mu(d\xi)\rangle)dt.
\]
\end{theorem}
Note that if the drift is of the particular form $B(x)=\beta
x+x\beta^\top$, where $\beta$ is a real $d\times d$ matrix, if
$b=\delta \alpha$ ($\delta\geq 0$) and in the absence of jumps, $X$
is a Wishart process (\cite{bru,alfonsi,pfaffel}).

\section{The Fourier-Laplace transform of affine processes}\label{sec: FLT}
Affine processes on positive semidefinite
matrices are defined in terms of the Laplace transform of their
transition probabilities, eq.~\eqref{def: affine process}. In general, the Laplace transform is a natural choice
for integral transform of generalized functions on proper cones such
as $S_d^+$.  However, Duffie, Filipovi\'c and
Schachermayer \cite{dfs} have defined affine processes on $\mathbb
R_+^m\times\mathbb R^n$ in terms of the exponentially affine form of
their characteristic function.  Only in the one dimensional case
$\mathbb R_+$, the two state-spaces coincide and therefore
also the two definitions of the affine property, either by the Laplace transform (Kawazu and Watanabe
\cite{KW}) or by the characteristic function.

Therefore, the question whether the
characteristic functions of a positive semidefinite affine process
is indeed exponentially affine in the state, is of considerable
interest. We will denote this property as being ``affine in the sense of Duffie, Filipovi\'c and Schachermayer''.

Unless the diffusion coefficient $\alpha$ vanishes, $X$ need not be
infinitely divisible, or equivalently, infinitely decomposable (for
the definition and characterization of these properties in the
affine Markov setting, see \cite[Definition 2.7, Example 2.8 and
Theorem 2.9]{cfmt}). This complicates the problem of extending the
affine formula eq.~\eqref{def: affine process} to the complex
domain, because it is not anymore guaranteed that the
Fourier-Laplace transform of $X$ exhibits no zeros, as is in the
infinite divisible case (\cite[Theorem 25.17]{Sato1999}) (which is a necessary condition
to write it in an exponentially affine way). From the
ODE perspective, there is a related technical problem, namely to
show that the real part of $\psi(t,u)$ as solution of the system of
generalized Riccati equations (eqs.~\eqref{eq: phi1}--\eqref{eq:
psi1} below) with imaginary initial data does not explode in finite
time. Indeed, we have the estimate
\[
\vert e^{-\phi(t,u)-\langle\psi(t,u),x\rangle}\vert\leq
e^{-\Real(\phi(t,u))-\langle\Real(\psi(t,u)),x\rangle}
\]
and if the real part of $\psi$ explodes in finite time, then the
characteristic function must have a zero  \footnote{We note that the $\phi$ coefficient
does not matter here: $\Real(\phi(t,u))\geq 0$ can be inferred from the specific form
of the generalized Riccati differential equations.}. In this section we extend
the affine transform formula to the full Fourier-Laplace transform,
under the premise that the diffusion component must be
non-degenerate or equals zero. For technical difficulties in the
degenerate case, see Remark \ref{rem: technical}.

We denote by $\mathcal S(S_d^+)$ the complex tube $S_d^++iS_d$,
and similarly $\mathcal S(S_d^{++})=S_d^{++}+i S_d$
and $\mathcal S(\mathbb R_+)=\mathbb R_++i \mathbb R$.
\begin{theorem}\label{th41}
Let $X$ be an affine process on $S_d^+$ $(d\geq 2)$, with a
diffusion coefficient $\alpha$ which is either invertible or zero. Then the
affine property \eqref{def: affine process} holds for all $t\geq 0$,
$x\in S_d^+$, and for all $u\in \mathcal S(S_d^{++})$, with
exponents $\phi: \re_+\times \mathcal S(S_d^{++}) \rightarrow
\mathcal S(\mathbb R_+)$ and $\psi: \re_+ \times \mathcal S(
S_d^{++}) \rightarrow \mathcal S(S_d^{++})$ which are the unique
global solutions of the generalized Riccati differential equations
\begin{align}\label{eq: phi1}
\partial_t\phi(t,u)&=\langle b,\psi(t,u)\rangle + c-\int\limits_{S_d^+\setminus\{0\}}\left(e^{-\langle \psi(t,u),\xi\rangle}-1\right)m(d\xi),\\\label{eq: psi1}
\partial_t\psi(t,u)&=-2\psi(t,u)\alpha\psi(t,u)+B^\top(\psi(t,u))+\gamma
\\\nonumber&\quad\quad-\int\limits_{S_d^+\setminus\{0\}}\left(e^{-\langle \psi(t,u),\xi\rangle}-1\right)\mu(d\xi)
\end{align}
given initial data $\phi(0,u)=0,\,\psi(0,u)=u$.
\end{theorem}

For the following two results we assume as in Theorem \ref{th41},
that $d\geq 2$, and the diffusion coefficient $\alpha$ of $X$ is either
invertible or zero.

\begin{theorem}\label{th42}
The affine property \eqref{def: affine process} also holds for
$u\in\mathcal S(S_d^+)$, with exponents $\phi: \re_+\times \mathcal
S(S_d^{+}) \rightarrow \mathcal S(\mathbb R_+)$ and $\psi: \re_+
\times\mathcal S(S_d^{+}) \rightarrow \mathcal S(S_d^{+})$, being
(not necessarily unique) solutions of the generalized Riccati
differential equations \eqref{eq: phi1}--\eqref{eq: psi1}.
\end{theorem}

Applying the above to $u\in iS_d$, we finally obtain:
\begin{corollary}\label{cor42}
$X$ is affine in the sense of Duffie, Filipovi\'c and Schachermayer
\cite{dfs}.
\end{corollary}
\subsection{Proof of Theorem \ref{th41}}
\begin{proof}
Let $u=v+iw\in \mathcal S(S_d^{++})$. We denote by $\psi(t,v)$ the
unique global solution of equation \eqref{eq: psi1} on $S_d^{++}$,
which exists due to \cite[Proposition 5.3]{cfmt}. $\psi(t,u)$ is
defined as maximal solution of \eqref{eq: psi1} on the open domain
$\mathcal S(S_d^{++})$. Note that the right side $R(\psi)$ of eq.~\eqref{eq: psi1} is an
analytic function thereon, hence it is in particular locally Lipschitz.
Accordingly, the maximal life time of $\psi(t,u)$ equals
\[
t_+(u):=\lim_{n\rightarrow\infty}\inf\{t>0\mid \Real\psi(t,u)\in\partial S_d^+\textrm{  or  }\|\psi(t,u)\|\geq n\}
\]
and we have $0<t_+(u)\leq \infty$.

First, we show that $\psi(t,u)$ does not touch the boundary of
$\mathcal S(S_d^{++})$ in finite time. To this end, we introduce the
function $\chi(t,u):=\Real(\psi(t,u+iv))$, which is well defined for
$t\in [0,t_+(u))$ and has values in $S_d^{++}$. Denote by $u\mapsto
R(u)$ the function on the right side of \eqref{eq: psi1}. By
straightforward inspection, one observes that for all $t<t_+(u)$
\[
\partial_t\chi(t,u)-\Real(R(\chi(t,u)))\succeq 0 =\partial_t\psi(t,v)-R(\psi(t,v)).
\]
Since $R$ is an analytic and quasi-monotone increasing function on $S_d^{++}$ with respect to the cone $S_d^+$ (see \cite[Definition 4.7 and Lemma 5.1]{cfmt}, we may invoke Volkmann's comparison result in the fashion of
\cite[Theorem 4.8]{cfmt} and derive
\[
\chi(t,u)\succeq \psi(t,v), \textrm{  for all  } t<t_+(u).
\]
But $\psi(t,v)\in S_d^{++}$, for all $t\geq 0$, \cite[Proposition 5.3]{cfmt}. Hence we have shown that
$\psi(t,u)$ does not touch the boundary of $\mathcal S(S_d^{++})$, which  is $\partial S_d^+\times i S_d$, in finite time, and therefore we have
\begin{equation}\label{only expl}
t_+(u):=\lim_{n\rightarrow\infty}\inf\{t>0\mid\,\,\|\psi(t,u)\|\geq n\}.
\end{equation}
Hence it remains to show that $\psi(t,u)$ does not explode in finite
time. Since affine transformations of the state space do not effect
the blow-up property, we may without loss of generality assume that
the diffusion coefficient equals zero or equals the identity matrix.
To obtain the necessary transformation, one can adapt
\cite[Propositions 4.13 and 4.14]{cfmt}. We introduce the shorthand
notation $K:=K_1+K_2$, where
\begin{equation*}
K_1(u):=\int\limits_{0<\|\xi\|\leq 1}\left(\int_0^1\langle
u,\xi\rangle e^{-s\langle u ,\xi\rangle}ds\right)\mu(d\xi)
\end{equation*}
and
\[
K_2(u):=\int\limits_{\|\xi\|> 1}\left(1- e^{-\langle u ,\xi\rangle}\right)\mu(d\xi)
\]
Using this decomposition, we can write
\[
R(u)=-2 u\alpha u+B^\top u+\gamma+K(u).
\]
Using the Cauchy-Schwarz inequality,  Lemma \ref{norm versus trace lemma} and condition \eqref{eq: mu}, we infer the existence of a constant $C_1\geq 0$ such that for all $u\in \mathcal S(S_d^+)$
\begin{equation}\label{K1}
\|K_1(u)\|\leq \|u\|\int_{S_d^+\setminus\{0\}}(\|\xi\|\wedge 1)\tr(\mu)(d\xi)=C_1\|u\|.
\end{equation}
The same condition allows to conclude the existence of a positive constant $C_2$ such that
\begin{equation}\label{K2}
\|K_2(u)\|\leq \int_{\|\xi\|>1} 2 \tr(\mu)(d\xi)=C_2<\infty
\end{equation}
where we once again have used Lemma \ref{norm versus trace lemma}.

By Lemma \ref{lemma est} in Appendix \ref{appb} we have that
\begin{equation}\label{eq: keyestimate}
\Real\langle
\overline{\psi}(t,u),\psi(t,u)\alpha\psi(t,u)\rangle\geq 0
\end{equation}
for all $t<t_+(u)$. Using estimates \eqref{K1}--\eqref{eq: keyestimate} and the Cauchy-Schwarz inequality, the existence of
a positive constant $C$ follows, such that for all $u\in\mathcal S(S_d^{++})$ and $t<t_+(u)$,
\begin{align*}
\partial_t \| \psi(t,u) \|^2&=2\Real \langle \overline{\psi}(t,u),
R(\psi(t,u)\rangle\\&\leq 2\Real\left\langle
\overline{\psi}(t,u),B^\top(\psi(t,u))+\gamma+K(\psi(t,u))\right\rangle\\&\leq
2C(1+\|\psi(t,u)\|^2).
\end{align*}
Hence, by Gronwall's Lemma (or,
equivalently, by standard comparison for scalar-valued ODEs) we
obtain for all $t<t_+(u)$,
\begin{equation}\label{bound for psi}
\|\psi(t,u)\|\leq e^{Ct}\sqrt{1+\|u\|^2}
\end{equation}
which in view of \eqref{only expl} proves that $t_+(u)=\infty$. So
we have shown that $t\mapsto\psi(t,u)$ is the global solution of
\eqref{eq: psi1} for all $u\in\mathcal S(S_d^+)$. Moreover,
$\Real(\psi(t,u))\in S_d^{++}$ for all $t\geq 0$ and the right side
of \eqref{eq: phi1} is well defined for all $u\in\mathcal S(S_d^+)$.
Therefore plugging $\psi(t,u)$ into \eqref{eq: phi1} and integrating
with respect to time yields $\phi(t,u)$.%,
%and it is complex analytic on the strip $u\in \mathcal S(S_d^{++}$.

Now for each $t>0, x\in S_d^+$, the Fourier-Laplace transform
\[
g(u)=\mathbb E[e^{-\langle u, X_t\rangle}\mid X_0=x]
\]
and the function
\[
f(u):=e^{-\phi(t,u)-\langle x,\psi(t,u)\rangle}
\]
are complex analytic functions on $\mathcal S(S_d^{++})$, and (in view
of \eqref{def: affine process}) they coincide on set of uniqueness,
namely $S_d^{++}$. Hence $f\equiv g$ on $\mathcal S(S_d^{++})$, which
proves the assertion.
\end{proof}
\subsection{Proof of Theorem \ref{th42}}
We can write $u=v+iw$, where $v\in S_d^+$ and denote for each $n\geq
1$ the matrix $u_n:=(v+\frac{1}{n}1)+iw$, where $1$ is the unit
$d\times d$ matrix. We further denote by $\psi_n(t)$ the solution of
\eqref{eq: psi1}, subject to $\psi_n(0)=u_n$, which exists globally
due to Theorem \ref{th41} because now $u_n\in\mathcal S(S_d^{++})$.

Let $\pi(x)$ be the projection of $x\in S_d$ onto $S_d^+$, which
exists uniquely, because $S_d^+$ is a closed convex set. For
$u=v+iw\in\mathcal S(S_d)$, we slightly abuse notation and write
\[
\pi(u):=\pi(v)+iw\in\mathcal S(S_d^+)
\]
Using the continuity of the right sides of \eqref{eq:
phi1}-\eqref{eq: psi1} we may as well consider $\phi_n(t)$ and $\psi_n(t)$ as
solutions to the generalized Riccati differential equations
\begin{align}\label{eq: phi1x}
\partial_t\phi_n(t,u)&=\langle b,\psi_n(t,u)\rangle + c-\int\limits_{S_d^+\setminus\{0\}}\left(e^{-\langle \pi(\psi_n(t,u)),\xi\rangle}-1\right)m(d\xi)\\\label{eq: psi1x}
\partial_t\psi_n(t,u)&=-2\psi_n(t,u)\alpha\psi_n(t,u)+B^\top(\psi_n(t,u))+\gamma
\\\nonumber&\quad\quad-\int\limits_{S_d^+\setminus\{0\}}\left(e^{-\langle \pi(\psi_n(t,u)),\xi\rangle}-1\right)\mu(d\xi)
\end{align}
subject to $\psi_n(0)=u_n$, $\phi_n(0)=0$ on the whole domain $\mathcal S(S_d)$.

Now by estimating \eqref{bound for psi} in the proof of Theorem
\ref{th41}, there exists a uniform constant $C$, such that for each
$n$,
\begin{equation}\label{bound for psi2}
\|\psi_n(t)\|\leq e^{Ct}\sqrt{1+\|u_n\|^2}.
\end{equation}
But this means that for any $T>0$, the family of curves
\[
\{\psi_n(t)\mid t\in[0, T]\}
\]
lie in a single compact set $K$. Since $T$ is arbitrary, an
application of Lemma \ref{l2} therefore yields that there exist
functions $t\mapsto\phi(t,u),t\mapsto\psi(t,u)$ on $[0,\infty)$
which are the pointwise limits of a subsequence
$(\phi_{n_k}(t),\psi_{n_k}(t))$ $(k\rightarrow\infty)$ and they
satisfy eqs.~\eqref{eq: phi1}--\eqref{eq: psi1}. Furthermore, we
have by dominated convergence,
\begin{align*}
e^{-\phi(t,u)-\langle \psi(t,u),x\rangle}&=\lim_{k\rightarrow\infty}
e^{-\phi_{n_k}(t)-\langle \psi_{n_k}(t),x\rangle}=\lim_{k\rightarrow\infty}\mathbb
E[e^{-\langle u_{n_k},X_t\rangle}\mid X_0=x]\\&=\mathbb E[e^{-\langle
u,X_t\rangle}\mid X_0=x].
\end{align*}
This ends the proof.
\begin{remark}\label{rem: technical}\rm
\begin{itemize}
\item It can easily be seen either by numerical experiments or explicit
calculations that (an appropriate adaption of) Lemma \ref{lemma est} does not hold, if $\alpha$
not equals a scalar multiple of the unit matrix. To be more precise,
in general, the real part of
\[
\tr(\bar x x \alpha x)=\tr(x\bar x x\alpha)
\]
can be strictly negative. For instance, using
\[
\alpha=\left(\begin{array}{ll} 1 & 0\\ 0& 0\end{array}\right),\quad x=\left(\begin{array}{ll} 1 & i\\ i& 4\end{array}\right)
\]
we obtain $\Real(x)\in S_2^+$, but $\Real\tr(\bar x x \alpha x)=-1<0$. As a consequence,
we cannot derive estimate \eqref{eq: keyestimate}, which in turn is
a technical necessity to obtain the a-priori estimate \eqref{bound
for psi} resp. \eqref{bound for psi2}. However, we conjecture that the problem concerning
degenerate, nonzero diffusion coefficient $\alpha$ admits the same
answer as that of Theorem \ref{th42}.
\item It should be noted, that the main technical complication
concerning jumps in prior research had been the presence of a
truncation function in the right side of eq.~\eqref{eq: psi1}. Only
the finding of Theorem \ref{th: maintheorem} let us establish the
general a-priori estimate \eqref{bound for psi} resp. \eqref{bound for psi2}.
\end{itemize}
\end{remark}
\subsection{Examples with degenerate, nonzero diffusion}\label{sec:
ex degenerate} In the presence of a non-zero diffusion component
$\alpha$, Theorem \ref{th41} requires that $\alpha$ is invertible.
It should, however, be reported that if $X$ is ``Wishart with state
independent jump behavior", then not only is it evident that $X$ is
affine in the sense of \cite{dfs}, but also the affine character of
the Laplace transform can be extended to the domain $\mathcal
S(S_d^+)$. And in this case, we can solve the Riccati equations
explicitly, with no non-degeneracy assumption on $\alpha$.
\begin{definition}\rm
A matrix-variate basic affine jump-diffusion $X$ on $S_d^+$ (MBAJD
in short) is an affine process with parameters
$\gamma=0,\,c=0,\,\mu\equiv 0$, a constant drift
\[
b=2p \alpha,\quad p\geq \frac{d-1}{2},
\]
and with a linear drift $B$ of the particular form
\[
B(x)=\beta x+x\beta^\top,
\]
where $\beta$ is a real $d\times d$ matrix.
\end{definition}
\begin{remark}
\begin{itemize}
\item If $d=1$, and $m(d\xi)$ is a multiple of the density of an
exponential distribution, then $X$ is a BAJD as introduced by Duffie
and Garleanu in \cite{dg}.
\item If $d\geq 2$ and $m\equiv 0$, then $X$ is a Wishart process, see
\cite{alfonsi}, \cite{bru} and \cite{pfaffel}.
\end{itemize}
\end{remark}
It is quite straightforward to check that any MBAJD is a
conservative Markov process (\cite[Remark 2.5]{cfmt}) and that
eqs.~\eqref{eq: phi1}--\eqref{eq: psi1} take the particular form
\begin{align*}
\partial_t\phi(t,u)&=2p\langle \alpha,\psi(t,u)\rangle -\int\limits_{S_d^+\setminus\{0\}}\left(e^{-\langle
\psi(t,u),\xi\rangle}-1\right)m(d\xi),
\\\partial_t\psi(t,u)&=-2\psi(t,u)\alpha\psi(t,u)+\psi(t,u)\beta+\beta^\top\psi(t,u).
\end{align*}
In the following we denote by $\omega^{\beta}_t$ the flow of the
vector field $\beta x+x\beta^\top$, that is,
\[
\omega^{\beta}:\,\, \mathbb R\times S_d^+\rightarrow
S_d^+,\quad\omega^{\beta}_t(x):=e^{\beta t}x e^{\beta^\top t}.
\]
Its twofold integral $\sigma^{\beta}_t: S_d^+\rightarrow S_d^+$ for
$t\geq 0$ is denoted by
\[
\sigma^{\beta}:\,\, \mathbb R_+\times S_d^+\rightarrow
S_d^+,\quad\sigma^{\beta}_t(x)=2\int_0^t \omega^{\beta}_s(x) ds.
\]
By matrix analysis, we obtain  the following semi-explicit solutions
for initial data $u\in \mathcal S(S_d^+)$,
\begin{align*}
\phi(t,u)&=p\log\det \left(I+u
\sigma_t^\beta(\alpha)\right)-\int\limits_{S_d^+\setminus\{0\}}\left(e^{-\langle
\psi(t,u),\xi\rangle}-1\right)m(d\xi),\\
\psi(t,u)&=e^{\beta^\top
t}\left(u^{-1}+\sigma_t^\beta(\alpha)\right)^{-1}e^{\beta t}.
\end{align*}
\begin{appendix}\label{appa}
\section{Convergence of ordinary differential equations}
The following results are consequences of standard ODE theory. The
first one is clearly elaborated in \cite[Lemma 8]{ODEpaper}, and the
second one is a variant of \cite[Lemma 9]{ODEpaper} and can be
proved similarly as in \cite{ODEpaper} (the difference being that we
drop the Lipschitz continuity of $f$, hence one cannot show that
every involved subsequence in \ref{l2} converges, let alone to the
same limit).

We recall them here for the convenience of the reader, and without
any proof. We consider the system of ordinary differential equations
on $\mathbb R^m$,
\begin{equation}\label{a1}
\partial_t\psi(t) = f(t,\psi(t)),\quad
\end{equation}
subject to an initial condition $\psi(0)=u\in \mathbb R^m$. Recall
that equation~(\ref{a1}) possesses a maximal solution on a half open
interval $[0, t_+(u))$ if the function $f\colon I\times U\to E$ is
continuous. Note however that such a solution may be not unique if
$f$ is not locally Lipschitz continuous.
\begin{lemma}\label{l1}
Let $U\subset E$ be open. Let $f,f_1, f_2,\ldots$ be continuous maps
from $I\times U$ to $E$. Suppose $f$ is locally Lipschitz and $f_n$
converge to $f$ uniformly on all compact subsets of $I\times U$. Let
$\psi_n\in C^1([0,\theta_n),U)$ be maximal solutions of
\begin{equation}\label{a2}
\partial_t\psi_n(t) = f_n(t,\psi_n(t))
\end{equation}
such that $\psi_n(0)$ converge to some $u\in U$ as $n\to \infty$.
Then we have
\begin{equation}\label{a3}
t_+(u) \leq \varliminf \theta_n.
\end{equation}
Let $0\leq a<t_+(u)$ and $n_0$ be such that $\theta_n>a$ for $n>
n_0$. Then the sequence $\psi_{n_0+k}(t)$, $k=1,2,\ldots$, converges
to $\psi(t)$ uniformly on $[0,a]$ as $k\to\infty$.
\end{lemma}
\begin{lemma}
\label{l2} Let $U\subset \mathbb R^m$ be open. Let $f,f_1,
f_2,\ldots$ be continuous maps from $I\times U$ to $\mathbb R^m$.
Suppose $f_n$ converge to $f$ uniformly on compact subsets of
$I\times U$. Let $0<a<T$ and $\psi_n\in C^1([0,a],U)$ be solutions
of~$(\ref{a2})$ such that $\psi_n(0)$ converge to some $u\in U$ as
$n\to \infty$. If for some compact set $K\subset U$, $\psi_n(t)\in
K$ for all $t\in [0,a]$, then there exists a (not necessarily
unique) solution $\psi(t)$ of equation~(\ref{a1}) on $[0,a]$, and a
subsequence $\psi_{n_k}(t)\to \psi(t)$ and
$\partial_t\psi_{n_k}(t)\to
\partial_t\psi(t)$ uniformly on $[0,a]$.
\end{lemma}
\section{A simple matrix inequality}\label{appb}
\begin{lemma}\label{lemma est}
For any complex valued $m\times n$ matrix $a$ and for any $b\in
\mathcal S(S_n^+)$ we have that
\[
\Real\tr(b\bar a^\top a)\geq 0
\]
\end{lemma}
\begin{proof}
Write $a=a_1+ia_2$, and $b=b_1+ib_2$. Then we have
\begin{align*}
\Real\tr(b \bar a^\top a)&=\Real\tr((b_1+ib_2)(a_1^\top -i a_2^\top)(a_1+ia_2))\\&=
\Real\tr((b_1+ib_2)(a_1^\top a_1+i a_1^\top a_2-i a_2^\top a_1+a_2^\top a_2))\\&=
\tr(b_1 (a_1^\top a_1))+\tr(b_1(a_2^\top a_2))+\tr(b_2 a_2^\top a_1)-\tr(b_2 a_1^\top a_2)\\&\geq
0+0+\tr(b_2 a_2^\top a_1)-\tr(b_2^\top a_1^\top a_2)=0.
\end{align*}
Here we have used that $a_1^\top a_1,\,b_1\in S_n^+$, $b_2=b_2^\top$ and the commutativity of the
matrix trace.
\end{proof}

\end{appendix}

\end{document}